\documentclass[english, 10pt]{amsart}
\usepackage{amsmath}
\usepackage{graphicx}
\usepackage{amssymb}
\usepackage{amsfonts, amsthm}
\usepackage[colorlinks=true,linkcolor=blue,citecolor=blue,urlcolor=blue]{hyperref}  

\newcommand{\al}{\alpha}
\newcommand{\be}{\beta}
\newcommand{\ga}{\gamma}
\newcommand{\de}{\delta}
\newcommand{\la}{\lambda}

\newcommand{\auskommentieren}[1]{}

\newcommand{\beq}{\begin{equation}}
\newcommand{\eeq}{\end{equation}}
\newcommand{\bea}{\begin{equation}\begin{aligned}}
\newcommand{\eea}{\end{aligned}\end{equation}}

\newtheorem{theorem}{Theorem}[section]

\newtheorem{lem}[theorem]{Lemma}
\newtheorem{cor}[theorem]{Corollary}

\theoremstyle{definition}
\newtheorem{defi}[theorem]{Definition}
\newtheorem{rem}[theorem]{Remark}

\usepackage{array} 
\usepackage{amssymb}
\usepackage{latexsym}
\usepackage{epic}
\usepackage{curves}
\usepackage{fancyhdr}
\usepackage{xspace}

\numberwithin{equation}{section}

\begin{document}
\title{Finite element approximation of Power Mean Curvature Flow}
\author{Heiko Kr\"oner}
\date{August 10, 2013}
\maketitle
\begin{abstract}
In \cite{S} the evolution of hypersurfaces in $\mathbb{R}^{n+1}$ with normal speed equal to a power $k>1$ of the mean curvature is considered and the levelset solution $u$ of the flow is obtained as the $C^0$-limit of a sequence $u^{\epsilon}$ of smooth functions solving the regularized levelset equations. 

We prove a rate for this convergence.

Then we triangulate the domain by using a tetraeder mesh and consider continuous finite elements, which are polynomials of degree $\le 2$ on each tetraeder of the triangulation. We show in the case $n=1$ (i.e. the evolving hypersurfaces are curves), that there are solutions $u^{\epsilon}_h$ of the above regularized equations in the finite element sense, which satisfy for every $0<\Theta<\frac{1}{2}$ an error estimate of the form
\beq
\|u-u^{\epsilon}_h\|_{C^{0, \Theta}}\le c \epsilon^{\la}+c \epsilon^{-\ga}h^{\de},
\eeq
where values for $\la, \ga, \de>0$ can be obtained explicitly.

Our method can be extended to the case $n>1$, if one uses higher order finite elements.
\end{abstract}

\tableofcontents

\footnotetext{
\textsc{Eberhard Karls Universit\"at, Mathematisches Institut, 
Auf der Morgenstelle 10,
D-72076 T\"ubingen,
Germany}\\
\textit {E-mail:} \href{mailto:kroener@na.uni-tuebingen.de}{kroener@na.uni-tuebingen.de} \\
\textit {Url:} \href{http://na.uni-tuebingen.de/~kroener/}{http://na.uni-tuebingen.de/$\sim$kroener/}}

\section{Introduction and main results} \label{intro}

The famous mean curvature flow, cf. e.g. \cite{ES} and \cite{HI2}, evolves hypersurfaces in the direction of their normal with normal speed equal to the mean curvature. This flow has---apart from being of great interest by itself--- important applications in image processing.
During the last thirty years many variants of extrinsic curvature flows have been analyzed, which differ mainly in the prescribed normal velocity and the ambient space, in which the evolution takes place, cf. e.g. the inverse mean curvature flow \cite{HI}, the Gauss curvature flow \cite{A} and the inverse mean curvature flow in a Lorentzian manifold \cite{G}.

Concerning the numerical analysis for these flows there exist results in the case of mean curvature flow, cf. e.g. \cite{DD}, \cite{DDE} and the references therein, and \cite{P} in the case of anisotropic mean curvature flow in higher codimension.
In \cite{D} K. Deckelnick proves a rate of convergence for the approximation of the levelset solution of mean curvature flow by using a finite difference scheme; for the approximation he uses the solution of the regularized levelset equation as an intermediate step and divides the error estimate correspondently into the approximation error between the levelset solution and the solution of the regularized levelset equation and the error for the finite difference approximation of the regularized levelset equation. See also \cite{M} for the former error estimate.

Recently F. Schulze \cite{S} considered the evolution of hypersurfaces in $\mathbb{R}^{n+1}$ in the direction of their normal, for which the normal speed is given by a power $k>1$ of the mean curvature.

To the author's knowledge there do not exist any numerical results for Schulze's flow \cite{S} so far. Our aim is to approximate the levelset solution of this flow using the method of finite elements and to prove a convergence rate. This is done similarly to \cite{D} by using the solution of the regularized levelset equation as an intermediate step. It will come out that a 'polynomial coupling' between the regularization parameter $\epsilon$ and the numerical parameter $h$ will ensure a polynomial convergence rate, cf. Theorem \ref{main_thm_endversion} and Corollary \ref{main_cor}.
In contrast to \cite{D} we use a levelset formulation as suggested in \cite{S}, for which the levelset function does not depend on the time, cf. (\ref{levelset_pmcf}) and \cite[equation (1.1)]{D}. This ensures that the nonlinearity coming from the exponent $k$ affects only lower order (spatial) derivatives of the levelset function.

We introduce our setting more precisely.

Let  $M$ be a smooth $n$-dimensional compact manifold without boundary, $k> 1$  and $x_0:M \rightarrow \mathbb{R}^{n+1}$ a smooth embedding such that $x_0(M)$ has positive mean curvature, then there exist a small $T>0$ and a smooth mapping
\beq
x:M \times [0, T) \rightarrow \mathbb{R}^n
\eeq
with
\bea \label{classical_pmcf}
x(0, \cdot) &= x_0 \\
\dot x(t,\xi) &= -H^k\nu.
\eea
Here, $H$ and $\nu$ denote the mean curvature and the outer normal of $x(t, \cdot)(M)$ at $x(t, \xi)$ respectively, cf. \cite[Section 1]{S}.

We call this a power mean curvature flow (PMCF).

We give a level set formulation of PMCF.
Let $\Omega \subset \mathbb{R}^{n+1}$ be open, connected and bounded having smooth boundary $\partial \Omega$ with positive mean curvature. We call the level sets $\Gamma_t =\partial\{x\in \Omega: u(x)>t\}$ of the continuous function $0 \le u \in C^0(\bar \Omega)$ a level set PMCF, if $u$ is a viscosity solution of 
\bea \label{levelset_pmcf}
div \left( \frac{D u}{|D u|} \right) =& -\frac{1}{|D u|^{\frac{1}{k}}} \\
u|_{\partial \Omega}=&0.
\eea
If $u$ is smooth in a neighborhood of $x\in \Omega$ with non vanishing gradient and satisfies there  (\ref{levelset_pmcf}), then the level set $\{u=u(x)\}$ moves locally at $x$  according to (\ref{classical_pmcf}).

Using elliptic regularization of level set PMCF we obtain the equation
\bea \label{regularized_levelset_pmcf}
div \left(\frac{D u^{\epsilon}}{\sqrt{\epsilon^2+|D u^{\epsilon}|^2}}\right) =& -(\epsilon^2+|D u^{\epsilon}|^2)^{-\frac{1}{2k}} \quad \text{in } \Omega \\
u^{\epsilon}=& 0\quad  \text{on } \partial \Omega,
\eea
which has unique smooth solutions $u^{\epsilon}$ for sufficiently small $\epsilon >0$; moreover, there is $c_0>0$ such that
\beq \label{12}
\|u^{\epsilon}\|_{C^1{(\bar \Omega})} \le c_0
\eeq
and (for a subsequence) 
\beq \label{15}
u^{\epsilon} \rightarrow u \in C^{0,1}(\bar \Omega)
\eeq
in $C^0(\bar \Omega)$. We call $u$ a weak solution of  (\ref{levelset_pmcf}), which is unique for $n\le 6$. 

All the above facts are proved in \cite[Section 4]{S}. 

A weak solution of (\ref{levelset_pmcf}) satisfies (\ref{levelset_pmcf})  in the viscosity sense, cf. Section \ref{Section_the_viscosity_solution_u}.
We formulate our first main result.

\begin{theorem} \label{main_result_neu}
For every $0<\Theta<1$, there is $0<\la=\la(\Theta, k)$ so that
\bea
\|u-u^{\epsilon}\|_{C^{0, \Theta}(\bar \Omega)} \le c \epsilon^{\la},
\eea
where $c=c(\Theta, k, \Omega)>0$ is a constant.
\end{theorem}

We need some notations before we formulate our second main result in Theorem \ref{main_thm_endversion}.
Let $\{\mathbb{T}_h: 0<h<h_0\}$ be a family  of regular triangulations of $\Omega$, $h$ the mesh size of  $\mathbb{T}_h$ and $h_0=h_0(\Omega)>0$ small, so that for each boundary tetraeder $T\in \mathbb{T}_h$ $n+1$ vertexes lie on $\partial \Omega$. We define
\beq
\Omega^h = \cup_{T\in \mathbb{T}_h}T;
\eeq
since $\Omega$ might lack convexity, there will not hold in general $\Omega^h\subset \bar \Omega$.
Let 
\beq
V_h:=\{w \in C^0(\bar \Omega^h): \forall_{T\in \mathbb{T}_h}w|_T \text{ polynom of degree} \le 2, \; w_{|\partial \Omega^h}=0\}.
\eeq

Let $d: \mathbb{R}^{n+1}\rightarrow \mathbb{R}$,  be the signed distance function of $\partial \Omega$ (sign convention so that $d_{|\Omega}<0$) and $\de_0 = \de_0(\Omega)>0$ small. 
For $0<\de<\de_0$ we define 
\beq
\Omega_{\de}= \{d<\de\}
\eeq
and have $\partial \Omega_{\de}\in C^{\infty}$, $\|\partial \Omega_{\de}\|_{C^{2}}\le c(\Omega)\|\partial \Omega\|_{C^{2}}$. Furthermore, 
there is a constant $0<\tilde c=\tilde c(\Omega)$ so that
\beq \label{851}
\partial \Omega^h \subset \Omega_{\tilde c h^2} \backslash \Omega_{-\tilde c h^2}. 
\eeq
We extend $u^{\epsilon}$ to a function in $C^m(\Omega_{\de_0})$ (and denote the extension by $u^{\epsilon}$ again), $m\in \mathbb{N}$ sufficiently large, so that 
\beq \label{865}
\|u^{\epsilon}\|_{C^m(\Omega_{\de_0})}\le c\|u^{\epsilon}\|_{C^m(\bar\Omega)},
\eeq   
 cf.  \cite[Lemma 1.2.16 in PDE II]{CG}.

\begin{theorem} \label{main_thm_endversion} Let $n=1$ and $n+1<\mu< 4$ and
 $1< \de < \frac{1}{2}+\frac{2}{\mu}$, then there exist $\be, \ga, c>0$ depending on $ \mu, \de, k, \Omega$, so that for every $0<\epsilon<\epsilon_0$, where $\epsilon_0>0$ small, and $h \le c\epsilon^{ \be}$ the equation
\beq \label{930}
\int_{\Omega^h} \frac{\left<D u^{\epsilon}_h, D \varphi_h\right>}{\sqrt{\epsilon^2+|D u^{\epsilon}_h|^2}} = \int_{\Omega^h}(\epsilon^2+|D u^{\epsilon}_h|^2)^{-\frac{1}{2k}}
\varphi_h \quad \forall \; {\varphi_h\in V_h },
\eeq
has a unique solution $u^{\epsilon}_h$ in 
\beq \label{912}
\bar B^h_{\rho}:=\{w_h\in V_h: \|w_h-u^{\epsilon}\|_{H^{1, \mu}(\Omega^h)} \le \rho \},
\eeq
where 
\beq
\rho=c\epsilon^{-\ga}h^{\de}.
\eeq
\end{theorem}

\begin{cor} \label{main_cor}
In the situation of Theorem \ref{main_result_neu} and Theorem \ref{main_thm_endversion}  holds for $0<\Theta<\frac{1}{2}$
\beq
\|u-u^{\epsilon}_h\|_{C^{0, \Theta}(\Omega^h)} \le c \epsilon^{\la} + c\epsilon^{- \ga}h^{ \de}.
\eeq
\end{cor}

\begin{rem} 
(i) Our proofs of Theorems \ref{main_result_neu} and \ref{main_thm_endversion} are constructive in the sense, that possible values for $\Theta,  \be,  \ga$ can be calculated explicitly.

(ii) An analogous result to Theorem \ref{main_thm_endversion} can be obtained in case $n>1$ by using higher order finite elements.
\end{rem}
The remaining part of the paper deals with the proof of the above Theorems.
In Section \ref{Section_the_viscosity_solution_u} we give the definition of a viscosity solution of (\ref{levelset_pmcf}), that accounts for the fact, that $Du$ might vanish. Therefore we adapt the definitions in \cite{ES} for the mean curvature flow (where a time dependent levelset function $u$ is used) to our situation.
Furthermore, we show that $u$ in (\ref{15}) is a viscosity solution of (\ref{levelset_pmcf}). 

In Section \ref{Section_u_minus_u_eps} we prove Theorem \ref{main_result_neu}, for what we modify an argument, which is used in \cite{D} to prove a corresponding result for the levelset formulation of mean curvature flow (in \cite{D} a time dependent levelset function is considered too). 

In Section \ref{higher_order_estimates} we derive higher order estimates for $u^{\epsilon}$ and Section \ref{Section_linear_equations} provides explicit constants in some estimates concerning linear equations, which are applied in Section \ref{Section_FE} in order to prove Theorem \ref{main_thm_endversion}. In Section \ref{Section_FE} we proceed similarly to \cite{FNP}, where the regularized levelset equation for the inverse mean curvature flow is approximated by finite elements; in contrast to our paper, cf. especially Corollary \ref{main_cor}, \cite{FNP} does not provide any quantitative information about the approximation error between the finite element solution and the original geometric problem, i.e. the solution of the not regularized levelset equation.

\section{The viscosity solution $u$} \label{Section_the_viscosity_solution_u}
We give the definition of a viscosity solution of (\ref{levelset_pmcf}) and prove that the limit $u$ in (\ref{15}) is a viscosity solution of (\ref{levelset_pmcf}). Both seems to be standard, but since these things are not carried out in \cite{S}, we present them for reasons of completeness here.

To define a viscosity solution of (\ref{levelset_pmcf}) we adapt the corresponding definitions in \cite[Sections 2.2 and 2.3]{ES} and \cite[Section 2]{CIL}.

By formal differentiation we get from (\ref{levelset_pmcf}) that
\beq \label{13}
F(u):= -|D u|^{\frac{1}{k}-1}(\de_{ij}-\frac{u_iu_j}{|Du|^2})u_{ij}=1.
\eeq

 and from (\ref{regularized_levelset_pmcf}) that
\beq \label{14}
F_{\epsilon}(u^{\epsilon}):=-(|D u^{\epsilon}|^2+\epsilon^2)^{\frac{1}{2k}-\frac{1}{2}}(\de_{ij}-\frac{u^{\epsilon}_iu^{\epsilon}_j}{|Du^{\epsilon}|^2+\epsilon^2})u^{\epsilon}_{ij}=1.
\eeq
We need the following definitions.
\begin{defi}
Let $u\in C^0(\Omega)$ and $\hat x \in \Omega$, then we define
\bea
J_{\Omega}^{2,+}u(\hat x) = & \{(p, X)\in  \mathbb{R}^{n+1} \times S(n+1) : 
u(x) \le u(\hat x)  
+ \left<p, x-\hat x\right> \\
& +\frac{1}{2}\left<X(x-\hat x), x-\hat x\right>
+ o(|x-\hat x|^2) \quad \text{as}\quad x\rightarrow \hat x\}
\eea
and for $x \in \Omega$
\bea
\bar J_{\Omega}^{2,+}u( x) =  \{(p, X)\in  \mathbb{R}^{n+1}  \times S(n+1) : 
\text{there are} \; x_k \in \Omega, \; \text{and} \\ \; (p_k, X_k)\in J_{\Omega}^{2,+}u(x_k),\; 
\text{so that}\; (x_k, p_k, X_k)\rightarrow ( x,p,X) \; \},
\eea
where $S(l)$, $l\in \mathbb{N}$, denotes the set of symmetric $n\times n$ matrices.
\end{defi} 
\begin{defi} \label{25}
(i) A continuous function $u: \Omega\rightarrow \mathbb{R}$ is a viscosity subsolution of (\ref{levelset_pmcf}), if for all  $(\eta,X) \in J^{2,+}_{\Omega}(u)(x)$, $x \in \Omega$, there holds
\beq
-|\eta|^{\frac{1}{k}-1}(\de_{ij}-\frac{\eta_i\eta_j}{|\eta|^2})X_{ij}\le1,
\eeq
if $\eta \neq 0$ and
\beq
-(\de_{ij}-\tilde \eta_i\tilde \eta_j)X_{ij}\le0
\eeq
for some $\tilde \eta$ with $|\tilde \eta|\le1$, if $\eta=0$.

(ii) A continuous function $u: \Omega\rightarrow \mathbb{R}$ is a viscosity supersolution of (\ref{levelset_pmcf}), if for all  $(\eta,X) \in J^{2,-}_{\Omega}(u)(x)$, $x \in \Omega$, there holds
\beq
-|\eta|^{\frac{1}{k}-1}(\de_{ij}-\frac{\eta_i\eta_j}{|D\eta|^2})X_{ij}\ge1,
\eeq
if $\eta \neq 0$ and
\beq
-(\de_{ij}-\tilde \eta_i\tilde \eta_j)X_{ij}\ge0
\eeq
for some $\tilde \eta$ with $|\tilde \eta | \le1$, if $\eta=0$.

(iii) A function $u$, which is supersolution and subsolution of (\ref{levelset_pmcf}) is a viscosity solution of (\ref{levelset_pmcf}).
\end{defi}

\begin{rem}
A simple inspection shows that we could have replaced $J^{2,+}_{\Omega}(u)(x)$ in the preceding definition by $\bar J^{2,+}_{\Omega}(u)(x)$ and $J^{2,-}_{\Omega}(u)(x)$ by $\bar J^{2,-}_{\Omega}(u)(x)$.
\end{rem}

Sometimes it is useful to have another definition available.
\begin{defi} \label{26}
(i) A function $u \in C^0(\Omega)$ is a viscosity subsolution of (\ref{levelset_pmcf}), provided that if
\beq
u-\varphi \text{ has a local maximum at a point } x_0 \in \Omega
\eeq
for each $\varphi \in C^{\infty}(\Omega)$, then
\beq
\begin{cases} \label{710}
-|D\varphi|^{\frac{1}{k}-1}(\de_{ij}-\frac{\varphi_i\varphi_j}{|D\varphi|^2})\varphi_{ij}\le1 \text{ at }x_0 \\
\text{if } D\varphi(x_0)\neq 0,
\end{cases}
\eeq
and
\beq
\begin{cases} \label{711}
-(\de_{ij}-\eta_i\eta_j)X_{ij}\le 0 \text{ at } x_0 \\
\text{for some }\eta\in \mathbb{R}^n \text{ with }|\eta| \le 1, \text{ if }D\varphi(x_0) =0.
\end{cases}
\eeq
(ii) A function $u \in C^0(\Omega)$ is a viscosity supersolution of (\ref{levelset_pmcf}), provided that if
\beq
u-\varphi \text{ has a local maximum at a point } x_0 \in \Omega
\eeq
for each $\varphi \in C^{\infty}(\Omega)$, then
\beq
\begin{cases}
-|D\varphi|^{\frac{1}{k}-1}(\de_{ij}-\frac{\varphi_i\varphi_j}{|D\varphi|^2})\varphi_{ij}\ge1 \text{ at }x_0 \\
\text{if } D\varphi(x_0)\neq 0,
\end{cases}
\eeq
and
\beq
\begin{cases}
-(\de_{ij}- \eta_i\eta_j)X_{ij}\ge 0 \text{ at } x_0 \\
\text{for some }\eta\in \mathbb{R}^n \text{ with }|\eta| \le 1, \text{ if }D\varphi(x_0) =0.
\end{cases}
\eeq
\end{defi}

\begin{theorem}
Definitions \ref{25} and \ref{26} are equivalent.
\end{theorem}

\begin{proof}
We only consider the case of viscosity subsolutions.

(i) We assume that $u$ is a viscosity subsolution according to Definition \ref{25}. Assume that $u-\varphi$ has a local maximum at a point $x_0\in \Omega$ for a fixed $\varphi \in C^{\infty}(\Omega)$. Hence for $x\in \Omega$ close to $x_0$ we get
\bea
u(x) \le& \varphi(x)+u(x_0) - \varphi(x_0) \\
=& u(x_0) + D\varphi(x_0)(x-x_0) + \frac{1}{2}D^2\varphi(x_0)(x-x_0)(x-x_0) + o(|x-x_0|^2),
\eea
which implies
\beq
(D\varphi(x_0), D^2\varphi(x_0)) \in J_{\Omega}^{2,+}(u)(x_0)
\eeq
and the claim follows.

(ii) We assume that $u$ is a viscosity subsolution according to Definition \ref{26}. Let $(\eta, X)\in J_{\Omega}^{2,+}(u)(x)$, $x\in \Omega$. Define for $\de>0$
\bea
\varphi_{\de}(y) = u(x) + \eta(y-x)+(y-x)^tX(y-x) + \de|y-x|^2
\eea
then $u-\varphi_{\de}$  has a local maximum in $x$. Hence (\ref{710}), (\ref{711}) hold with
\beq
D\varphi_{\de}(x)=\eta, \quad D^2\varphi_{\de}(x)=X+\de E.
\eeq
Letting $\de\rightarrow 0$ proves the claim.
\end{proof}

\begin{lem}
The function $u$ in (\ref{15}) is a viscosity solution of (\ref{levelset_pmcf}).
\end{lem}
\begin{proof} We adapt \cite[Section 4.3]{ES}.
Let $\varphi \in C^{\infty}(\Omega)$ and suppose $u-\varphi$ has a {\it strict} local maximum at a point $x_0 \in \Omega$. As 
$u^{\epsilon}\rightarrow  u$ uniformly, $u^{\epsilon}-\varphi$ has a local maximum at a point $x_{\epsilon}\in \Omega$ with
\beq \label{19}
x_{\epsilon} \rightarrow x_0 \quad \text{as }\epsilon \rightarrow 0.
\eeq
Since $u^{\epsilon}$ and $\varphi$ are smooth, we have
\beq
Du^{\epsilon} = D \varphi, \quad D^2u^{\epsilon} \le D^2 \varphi \quad \text{at } x_{\epsilon}.
\eeq
Thus (\ref{14}) implies
\beq \label{18}
-(|D \varphi|^2+\epsilon^2)^{\frac{1}{2k}-\frac{1}{2}}(\de_{ij}-\frac{\varphi_i\varphi_j}{|D\varphi|^2+\epsilon^2})\varphi_{ij}\le1
\quad \text{ at } x_{\epsilon}.
\eeq
Suppose first $D\varphi(x_0) \neq 0$. Then $D\varphi(x_{\epsilon}) \neq 0$ for small $\epsilon>0$. We consequently may pass to limits in (\ref{18}), recalling (\ref{19}) to deduce 
\beq \label{21}
-|D \varphi|^{\frac{1}{k}-1}(\de_{ij}-\frac{\varphi_i\varphi_j}{|D\varphi|^2})\varphi_{ij}\le 1 \quad \text{at } x_0.
\eeq
Next, assume instead $D\varphi(x_0)=0$. Set
\beq
\eta^{\epsilon}:= \frac{D\varphi(x_{\epsilon})}{(|D\varphi(x_{\epsilon})|^2+\epsilon^2)^{\frac{1}{2}}}
\eeq
so that (\ref{18}) becomes
\beq \label{20}
-(\de_{ij}-\eta^{\epsilon}_i\eta^{\epsilon}_j)\varphi_{ij}\le
(|D \varphi|^2+\epsilon^2)^{\frac{1}{2}-\frac{1}{2k}} \quad \text{at } x_{\epsilon}.
\eeq 
Since $|\eta^{\epsilon}|\le 1$, we may assume, upon passing to a subsequence and relabeling if necessary, that $\eta^{\epsilon}\rightarrow \eta$ in $\mathbb{R}^n$ for some $|\eta|\le 1$. Sending $\epsilon$ to $0$ in (\ref{20}) we discover
\beq \label{22}
-(\de_{ij}-\eta_i\eta_j)\varphi_{ij}\le 0 \quad \text{at } x_0.
\eeq

If $u-\varphi$ has a local maximum, but not necessarily a strict maximum at $x_0$, we repeat the argument above with $\varphi(x)$ replaced by 
\beq
\tilde \varphi(x) = \varphi(x) + |x-x_0|^4,
\eeq
again to obtain (\ref{21}) or (\ref{22}).

Consequently, $u$ is a weak subsolution. That $u$ is a weak supersolution follows analogously.
\end{proof}

\section{Estimate of $u-u^{\epsilon}$} \label{Section_u_minus_u_eps}
We first define some constants, which will determine an error estimate for $u^{\epsilon}-u$, as will become clear in the succeeding Theorem \ref{24}.

Let
\beq \label{9}
\ga >1+k
\eeq
and $\al, s>0$ be small so that
\beq \label{10}
\be_1(\al, s) > \be_2(\al, s),
\eeq
where 
\bea \label{23_}
\be_1(\al, s) :=\frac{2-s+\al(2-\frac{1}{k})}{\ga(2-\frac{1}{k})+\frac{1}{k}-1}, \quad 
\be_2(\al, s):=\frac{\al+ks}{\ga-k-1}
\eea
and choose 
\beq \label{714}
0<r<\frac{\al}{\ga}.
\eeq 

\begin{theorem} \label{24}
There is $c=c(k, \Omega)>0$ such that
\beq
\|u^{\epsilon}-u\|_{C^0(\bar \Omega)} \le c \epsilon^{\min(r,s)}
\eeq
for all $\epsilon>0$.
\end{theorem}
\begin{cor}
By interpolation we get in the situation of Theorem \ref{24} for $0<\theta<1$ that
\beq
[u^{\epsilon}-u]_{\theta, \Omega} \le c(\theta, k, \Omega)\epsilon^{\min(r,s)(1-\theta)} \quad \forall \epsilon>0,
\eeq
where the bracket denotes the H\"older semi-norm.
\end{cor}

\begin{rem} \label{712} We explain how we can deduce  rates of convergence explicitly.
Since inequality (\ref{10}) 'improves'  for decreasing $s>0$ we choose $s=\frac{\al}{\ga}$ in view of (\ref{714}). Then we maximize $\frac{\al}{\ga}$ with respect to $\al, \ga$ under the constraints $\al>0$,(\ref{9}) and (\ref{10}). This can be done by assuming equality in (\ref{10}) and solving this equation for $\al$, which is possible; then it suffices to maximize a nonlinear expression for $\ga$ under the constraint that the weak inequality $\ge$ holds in (\ref{9}). Small perturbations if necessary of maximizers of the latter optimization problem lead to feasible values for $r,s$. 
\end{rem}
In the remaining part of this section we prove Theorem \ref{24} by adapting  the proof of \cite[Theorem 1.2]{D}.

For $\epsilon>0$ we define $w_{\epsilon}:\bar \Omega \times \bar \Omega \rightarrow \mathbb{R}$ by
\beq
w_{\epsilon} (x,y):=\mu u(x)-u^{\epsilon}(y)-\frac{\epsilon^{-\al}}{\ga}|x-y|^{\ga}, \quad x,y \in \bar \Omega,
\eeq
where 
\beq \label{7}
\mu=\mu(\epsilon)=(1-\epsilon^s)^k.
\eeq
We use the abbreviation
\beq \label{11}
\varphi(x,y) := \frac{\epsilon^{-\al}}{\ga}|x-y|^{\ga}.
\eeq

Let $\hat x, \hat y \in \bar \Omega$ such that
\beq
w_{\epsilon}(\hat x, \hat y)= \sup_{\bar \Omega\times \bar \Omega}w.
\eeq
\begin{lem} \label{lemma1}
There holds $\hat x \in \partial \Omega$ or $\hat y \in \partial \Omega$.
\end{lem}
\begin{proof}
We assume $\hat x, \hat y \in \Omega$.  From \cite[Theorem 3.2]{CIL} we deduce that for every $\rho>0$ there are $X,Y \in S(n+1)$ such that
\beq \label{1}
(D_x\varphi(\hat x, \hat y), X) \in \bar J_{\Omega}^{2,+}(\mu u)(\hat x) \quad \wedge \quad (D_y\varphi(\hat x, \hat y), Y) \in \bar J_{\Omega}^{2,+}(-u^{\epsilon})(\hat y)
\eeq
and
\beq \label{2}
-(\frac{1}{\rho}+\|A\|)I 
\le \begin{pmatrix} X & 0\\ 0 & Y   \end{pmatrix}
\le A + \rho A^2,
\eeq
where $A:= D^2\varphi(\hat x, \hat y)$. We calculate
\beq  \label{8}
D_x\varphi(\hat x, \hat y)=\epsilon^{-\al}|\xi|^{\ga-2}\xi=-D_y\varphi(\hat x, \hat y), \quad \xi=\hat x-\hat y,
\eeq
and
\beq \label{30}
A=\begin{pmatrix} B & -B\\ -B & B   \end{pmatrix}, \quad B = \epsilon^{-\al}|\xi|^{\ga-4}((\ga-2)\xi\otimes\xi+|\xi|^2I).
\eeq
Using
\beq
F(\mu u)=\mu^{\frac{1}{k}}, \quad F_{\epsilon}(-u^{\epsilon})=-1,
\eeq
we conclude from (\ref{1}) that
\beq \label{3}
-(\de_{ij}-\frac{D_{x^i}\varphi D_{x^j}\varphi}{|D_x\varphi|^2})X_{ij} \le \mu^{\frac{1}{k}}|D_x\varphi|^{1-\frac{1}{k}} \text{ at } \hat x 
\eeq
if $D\varphi(\hat x) \neq 0$ and
\beq \label{3_neu}
-(\de_{ij}- \eta_i\eta_j)X_{ij}\le 0 \text{ at } \hat x \\
\eeq
for some $\eta\in \mathbb{R}^n \text{ with }|\eta| \le 1$ if $D\varphi(\hat x)=0$;
furthermore, there holds
\beq \label{3_neu_neu}
-(\de_{ij}-\frac{D_{y^i}\varphi D_{y^j}\varphi}{|D_y\varphi|^2+\epsilon^2})Y_{ij} \le -(|D_y\varphi|^2+\epsilon^2)^{\frac{1}{2}-\frac{1}{2k}}.
\eeq
From (\ref{2}) we get for all $\zeta \in \mathbb{R}^n$
\bea
\zeta^t(X+Y)\zeta =& (\zeta^t, \zeta^t)\begin{pmatrix} X & 0\\ 0 & Y   \end{pmatrix}\begin{pmatrix} \zeta\\ \zeta   \end{pmatrix} \\
\le& (\zeta^t, \zeta^t)\left\{\begin{pmatrix} B & -B\\ -B & B   \end{pmatrix}+2\rho\begin{pmatrix} B^2 & -B^2\\ -B^2 & B^2   \end{pmatrix}\right\}\begin{pmatrix} \zeta\\ \zeta   \end{pmatrix} \\
=& 0,
\eea
i.e. 
\beq \label{4}
X+Y \le 0,
\eeq
and
\bea \label{5_}
\xi^tY\xi = & (0, \xi^t)\begin{pmatrix} X & 0\\ 0 & Y   \end{pmatrix}\begin{pmatrix} 0\\ \xi   \end{pmatrix}\\
\le& \xi^tB\xi + 2 \rho \xi^tB^2\xi \\
\le& (\ga-1)\epsilon^{-\al}|\xi|^{\ga}+2 \rho \xi^tB^2\xi. 
\eea

{\it Case $\hat x \neq \hat y$:}
We add the inequalities (\ref{3}) and (\ref{3_neu_neu}) and get
\bea \label{5}
LHS:=& -(\de_{ij}-\frac{D_{x^i}\varphi D_{x^j}\varphi}{|D_x\varphi|^2})X_{ij}-(\de_{ij}-\frac{D_{y^i}\varphi D_{y^j}\varphi}{|D_y\varphi|^2+\epsilon^2})Y_{ij}  \\
\le&   (\mu^{\frac{1}{k}}-1)|D_x\varphi|^{1-\frac{1}{k}}. 
\eea
We estimate $LHS$ from below
\bea \label{6}
LHS =& -(\de_{ij}-\frac{D_{x^i}\varphi D_{x^j}\varphi}{|D_x\varphi|^2})(X_{ij}+Y_{ij}) \\
& -\epsilon^2 \frac{D_{x^i} \varphi D_{x^j} \varphi}{|D_x\varphi|^2(|D_y\varphi|^2+\epsilon^2)}Y_{ij} \\
 \ge& -\frac{\epsilon^2 \xi^tY\xi}{|\xi|^2(|\xi|^{2\ga-2}\epsilon^{-2\al}+\epsilon^2)} \\
 \ge& \frac{-(\ga-1)\epsilon^{2-\al}|\xi|^{\ga}-2\epsilon^2\rho\xi^tB^2\xi}{|\xi|^2(|\xi|^{2\ga-2}\epsilon^{-2\al}+\epsilon^2)}
\eea
where we used (\ref{4}) and (\ref{5_}). Combining (\ref{5}) with (\ref{6}), letting $\rho \rightarrow 0$ and applying the relations (\ref{7}) and (\ref{8}) yield
\beq
-\frac{(\ga-1)\epsilon^{2-\al}|\xi|^{\ga-2}}{|\xi|^{2\ga-2}\epsilon^{-2\al}+\epsilon^2} \le -\epsilon^{s-\al(1-\frac{1}{k})}|\xi|^{(\ga-1)(1-\frac{1}{k})}.
\eeq
We multiply this inequality by the denominator of the left-hand side and deduce two inequalities
\bea
-(\ga-1)\epsilon^{2-\al}|\xi|^{\ga-2}\le&
 -\epsilon^{s-\al(3-\frac{1}{k})}|\xi|^{(\ga-1)(1-\frac{1}{k})+2\ga-2} \\
-(\ga-1)\epsilon^{2-\al}|\xi|^{\ga-2}\le&
 -\epsilon^{s+2-\al(1-\frac{1}{k})}|\xi|^{(\ga-1)(1-\frac{1}{k})},
\eea
which lead to 
\bea
(\ga-1)\epsilon^{2-s+\al(2-\frac{1}{k})} \ge& |\xi|^{\ga(2-\frac{1}{k})-1+\frac{1}{k}} \\
(\ga-1)^k\epsilon^{-\al-ks} \ge& |\xi|^{-\ga+k+1}.
\eea
Accounting for (\ref{9}) we have
\bea \label{23}
|\xi| \le& (\ga-1)^{\frac{1}{\ga(2-\frac{1}{k})-1+\frac{1}{k}}} \epsilon^{\frac{2-s+\al(2-\frac{1}{k})}{\ga(2-\frac{1}{k})+\frac{1}{k}-1}}=:c_1\epsilon^{\be_1(\al, s)} \\
|\xi|\ge& (\ga-1)^{\frac{k}{-\ga+k+1}}\epsilon^{\frac{\al+ks}{\ga-k-1}}=:c_2\epsilon^{\be_2(\al, s)}.
\eea
In view of (\ref{10}) we get a contradiction for small $\epsilon>0$.

{\it Case $\hat x = \hat y$:}  
Due to $\ga>2$ and (\ref{30}) we have $B=0$, so that a calculation as in (\ref{5_}) (now with $\eta$ instead of $\xi$) shows
\beq
\eta^tY\eta \le 0.
\eeq
Hence, adding (\ref{3_neu}) to (\ref{3_neu_neu}) and having (\ref{4}) in mind we get
\bea
\epsilon^{1-\frac{1}{k}} \le& (\de_{ij}-\eta_i\eta_j)X_{ij} + \de_{ij}Y_{ij} \\
\le& (\de_{ij}-\eta_i\eta_j)(X_{ij} + Y_{ij}) + \eta^tY\eta \\
\le& 0,
\eea
which is a contradiction.
\end{proof}
\begin{lem}
There is $c_4>0$ such that
\beq
w_{\epsilon}(\hat x, \hat y) \le c_4 \epsilon^{r}.
\eeq
\end{lem}
\begin{proof}
In view of Lemma \ref{lemma1} we can assume in equation (\ref{11}) w.l.o.g. that $\hat y\in \partial \Omega$. Hence we can write
\beq
w_{\epsilon}(\hat x, \hat y)= \mu u(\hat x)-\mu u(\hat y)-\frac{\epsilon^{-\al}}{\ga}|\hat x-\hat y|^{\ga}.
\eeq
In case $|\hat x-\hat y|\le \epsilon^r$ we get using the lipschitz continuity of $u$
\beq
w_{\epsilon}(\hat x, \hat y) \le \mu c_0|\hat x-\hat y| \le \mu c_0 \epsilon^r,
\eeq
which proves the lemma. 

The remaining case $|\hat x-\hat y|> \epsilon^r$ is not available for sufficiently small $\epsilon>0$, for we estimate
\beq
w_{\epsilon}(\hat x, \hat y) \le 2\mu c_0 - \frac{\epsilon^{r\ga-\al}}{\ga}\rightarrow -\infty, \quad \epsilon \rightarrow 0.
\eeq 
\end{proof}
Now, collecting facts we finish the estimate for $u-u^\epsilon$.
Let $x \in \Omega$ arbitrary. Then
\bea
u(x)-u^{\epsilon}(x)=& \mu u(x)-u^{\epsilon}(x)+(1-\mu)u(x) \\
=& w_{\epsilon}(x,x)+(1-\mu)u(x)\\
\le& c_4\epsilon^{r} +c_0\epsilon^s\\
\le& c_5 \epsilon^{\min(r,s)},
\eea
with a positive constant $c_5$. Interchanging the roles of $u$ and $u^{\epsilon}$ we see, that there is a positive constant $c_6$ with
\beq
|u(x)-u^{\epsilon}(x)| \le c_6 \epsilon^{\min(r,s)}.
\eeq

\section{Higher order estimates of $u^{\epsilon}$} \label{higher_order_estimates}

In this section we make the $\epsilon$-dependence of a bound for higher order derivatives of $u^{\epsilon}$ explicit. 

We recall that the $u^{\epsilon}$ are $C^{\infty}$, bounded $\|u^{\epsilon}\|_{C^1(\bar \Omega)}\le c_0$ and satisfy the quasilinear equations in divergence form
\beq \label{103}
-D_ia^i(Du^{\epsilon}) = f, \quad {u^{\epsilon}}_{|\partial \Omega}=0,
\eeq
where 
\beq
a^i(p) = \frac{p^i}{\sqrt{\epsilon^2+|p|^2}}, \quad p \in \mathbb{R}^d,
\eeq
 and
 \beq
 f = -(\epsilon^2+|D u^{\epsilon}|^2)^{-\frac{1}{2k}}.
 \eeq 
Let us denote
\beq
a^{ij}(p) := \frac{\partial a^i}{\partial p_j}(p) = \frac{\epsilon^2\de^{ij}+ |p|^2\de^{ij}-p^ip^j}{(\epsilon^2+ |p|^2)^{\frac{3}{2}}},
\eeq
the largest and smallest eigenvalue of $a^{ij}(p)$ by $\Lambda(p)$ and $\lambda(p)$, respectively, and $\Lambda = \sup_{\bar B_{c_0}(0)}\Lambda(p)$, $\lambda = \inf_{\bar B_{c_0}(0)}\lambda(p)$.
In $\bar B_{c_0}(0)\subset \mathbb{R}^{n+1}$ we have
\beq \label{104}
0<c\epsilon^2 \de^{ij} \le a^{ij} \le \frac{c}{\epsilon}\de^{ij}, \quad \frac{\Lambda(p)}{\lambda(p)} \le \frac{c}{\epsilon^2}, \quad \frac{\Lambda}{\lambda} \le \frac{c}{\epsilon^3}.
\eeq
From standard $L^2$-regularity theory of quasilinear equations in divergence form we get, see for example the proof of \cite[Theorem 1.5.1 in PDE II]{CG}, that all second derivatives of $u^{\epsilon}$ except for the second derivative in normal direction at the boundary are bounded in the $L^2$-norm by
\beq
\frac{c}{\epsilon^2}\|f\|_{L^2(\Omega)} + \frac{c}{\epsilon^{\frac{3}{2}}}c_0\le \frac{c}{\epsilon^{2+\frac{1}{k}}}.
\eeq
Hence
\beq
\|u^{\epsilon}\|_{H^{2,2}(\Omega)} \le \frac{c}{\epsilon^{4+\frac{1}{k}}}
\eeq
and bounds for higher order derivatives of $u^{\epsilon}$ are obtained iteratively.

\section{Tracking constants in linear equations} \label{Section_linear_equations}
  
We consider linear equations of the form
\beq \label{201}
Lu = D_i(a^{ij}D_ju) + c^iD_i u = g+D_if^i,
\eeq
in $\tilde \Omega$, where we assume that
\beq \label{864}
\la >0, \quad a^{ij}\ge \la \de^{ij}, \quad \sum |a^{ij}|^2 \le \Lambda^2, \quad \la^{-2}\sum |c^i|^2 \le \nu^2
\eeq
and $\tilde \Omega=\Omega^h$, $0<h<h_0$, or $\tilde \Omega=\Omega_{\de}$, $0<\de<\de_0$. In the following results constants are uniform  with respect to $h, \de$.

Our aim in the present section is to  provide Corollary \ref{823}, which will be needed in Section \ref{Section_FE}.  We assume in this section $\la<1<\nu$.
\begin{theorem} \label{210}
Let $f^{i}\in L^q(\tilde \Omega)$, $g \in L^{\frac{q}{2}}(\tilde \Omega)$, $q>n+1$. Then if $u\in H^{1,2}(\tilde \Omega)$ is a 
subsolution (supersolution) of 
\beq
Lu = g + D_if^i
\eeq
in  $\tilde \Omega$ satisfying $u\le 0 (\ge 0)$ on $\partial \tilde \Omega$, we have
\beq
\sup_{\tilde \Omega} u (-u) \le C (\|(u^+(u^-)\|_{L^2(\tilde \Omega)}+R),
\eeq
where $R = \la^{-1}(\|f\|_{L^q(\tilde \Omega)}+\|g\|_{L^{\frac{q}{2}}(\tilde \Omega)})$ and
\beq
C = C(n, q, |\tilde \Omega|)\nu^{\frac{3q}{2(q-1)}}.
\eeq
\end{theorem}
\begin{proof}
A careful view of the proof of \cite[Theorem 8.15] {GT} shows the claim.
\end{proof}
\begin{theorem} \label{212}
In the situation of Theorem \ref{210} holds
\beq
\sup_{\tilde \Omega} u (-u) \le \sup_{\partial \tilde \Omega}u^+(u^-)+CR,
\eeq
where $R = \la^{-1}(\|f\|_{L^q(\tilde \Omega)}+\|g\|_{L^{\frac{q}{2}}(\tilde \Omega)})$ and
\beq
C = C(n, q, |\tilde \Omega|)(1+\nu^{\frac{5q-2}{2q-2}}) .
\eeq
\end{theorem}
\begin{proof}
Use Theorem \ref{210} and the proof of \cite[Theorem 8.16] {GT}. 
\end{proof}
 
 Let $q>n+1$.
 \begin{lem} \label{220}
Let  $g \in L^2(\tilde \Omega)$ and $f^i \in L^{q}(\tilde \Omega)$ then there exists a unique solution $u \in H^{1,2}_0(\tilde \Omega)$ of (\ref{201}) and there holds
\bea \label{241}
\|Du\|_{L^2(\tilde \Omega)} \le&  c_2(\|f\|_{L^{q}(\tilde \Omega)} +\|g\|_{L^2(\tilde \Omega)}),
\eea
where
\beq
c_2:=c_1(q) (\nu+ \frac{1}{\sqrt{\la}})\frac{1}{\la} \nu^{\frac{5q-2}{2q-2}}.
\eeq
\end{lem}
\begin{proof}
Use $u$ as test function, apply standard estimates and Theorem \ref{212}.
\end{proof}
 
Let $\hat \Omega= \Omega_{\de}$, $0<\de<\de_0$ arbitrary but fixed, then there holds the following Lemma with constants being uniform in $\de$.  
\begin{lem} \label{213}
Let $u \in H^{1,2}_0(\hat \Omega)$ be the solution of (\ref{201}) with $f^i=0$, $g \in L^2(\hat \Omega)$ and
\beq \label{217}
a^{ij}, c^i \in C^{1}(\bar {\hat \Omega}), \quad \|Da^{ij}\|_{C^0(\bar {\hat \Omega})}+ \|Dc^i\|_{C^0(\bar {\hat \Omega})} \le a_1.
\eeq
Then there holds
\beq \label{230}
\|u\|_{H^{2,2}(\hat \Omega)} \le  c_3 \|g\|_{L^2(\hat \Omega)},
\eeq
where 
\beq
c_3:=cc_1(q)c_2\Lambda (\frac{a_1}{\la^2}+\frac{\nu}{\la}).
\eeq
\end{lem}
\begin{proof}
The proof is a straight forward calculation.
\end{proof}

We have the following theorem.
\begin{theorem} \label{221}
We assume the situation of Lemma \ref{220} with $\tilde \Omega =\Omega^h$ and
\bea \label{817}
 a^{ij}, c^i \in C^{1}(\bar \Omega_{\de_0}), \quad \|Da^{ij}\|_{C^0(\bar \Omega_{\de_0})}+ \|Dc^i\|_{C^0(\bar \Omega_{\de_0})} \le a_1.
\eea
Let $u$ be the unique solution of  (\ref{201}) in $\Omega^h$.
Then for 
\beq \label{869}
0<h \le h_0:=\min(\frac{\nu^2c_7^2}{4}, (\frac{\de_0}{4})^{2q}),
\eeq
cf. (\ref{861}),  (\ref{910}) and (\ref{911}), there exists a unique FE solution $u_h\in V_h$ of  (\ref{201}) in $\Omega^h$, we have
\beq \label{240}
\|u-u_h\|_{H^{1,2}(\Omega^h)} \le c_4 \inf_{v_h\in V_h}\|u-v_h\|_{H^{1,2}(\Omega^h)}.
\eeq
and 
\beq \label{242}
\|u_h\|_{H^{1,2}(\Omega^h)} \le c_5 \|u\|_{H^{1,2}(\Omega^h)},
\eeq
where
\beq
c_4:=c(\frac{\Lambda}{\la}+\nu), \quad
c_5:=c_4+1. 
\eeq
\end{theorem}

\begin{cor} \label{823}
In the situation of Theorem \ref{221} holds
\beq
\|u_h\|_{H^{1,2}(\Omega^h)} \le c_8(\|(f\|_{L^{q}(\Omega^h)} +\|g\|_{L^{2}(\Omega^h)})
\eeq
with $c_8:=c_2c_5$.
\end{cor}

To prove Theorem  \ref{221} we would like to apply the Schatz argument, cf. \cite[Theorem 5.7.6]{BS} or \cite{Sc}, which uses the adjoint operator $L^{*}$ given by
\beq
L^{*} : H^{1,2}_0(\tilde \Omega) \rightarrow  H^{-1,2}(\tilde \Omega), \quad L^{*}u = D_i(a^{ij}D_ju - c^iu), 
\eeq
i.e.
\beq
\left<Lu,v\right>_{H^{-1}, H^1} = \left<u,L^{*}v\right>_{H^{1}, H^{-1}} \quad \forall u, v \in H^{1,2}_0(\tilde \Omega),
\eeq 
and needs that in the situation $\tilde \Omega = \Omega^h$ the space $H^{1,2}_0(\Omega^h)$ lies in the image of  $L^{*}$ and---for our case--- that $L^2$-estimates for $L^{*}$ with explicit constants are available. 
But both is not ensured, because $D_ic^i$ does not have the right sign necessarily and $\partial \Omega^h$ might lack the needed regularity (e.g. $\partial \Omega^h\in C^{0,1}$ and $\Omega^h$ convex). 

In the remaining part of this section we prove Theorem  \ref{221} using a modified Schatz argument.

In view of \cite[Theorem 8.6]{GT} there exists a countable set $\Sigma \subset \mathbb{R}$ so that for all $\sigma \notin \Sigma$ and all $g \in L^2(\tilde \Omega)$ there exists a unique solution $u \in H^{1,2}_0(\tilde \Omega)$ of the equation
\beq
(L^{*} + \sigma) u = g.
\eeq
$\Sigma$ depends on $h$ and $\de$, and in the following we will only use, that for  $h$ and $\de$ fixed the corresponding $\mathbb{R}\backslash\Sigma$ has 0 as accumulation point.

\begin{proof}[Proof of Theorem \ref{221}]

(i) Let $u$ be the unique solution of (\ref{201}) in $\Omega^h$. We assume that $u_h$ is a FE solution of (\ref{201}) in $V_h$ and extend $u, u_h$ by 0 to $\mathbb{R}^{n+1}$. Set $\de = h^{\frac{3}{2}}$, then
\beq	
\forall_{0 <h<h_0}  \quad \partial \Omega^h\subset \Omega_{\frac{\de}{2}}\backslash \Omega_{-\frac{\de}{2}}, 
\eeq
cf. (\ref{851}).

Choose a positive 
\beq \label{824}
\frac{1}{4c_3}> \sigma \notin \Sigma
\eeq
and let $w\in H^{1,2}_0(\Omega_{\de})$ be the unique solution of 
\beq \label{231}
(L^{*}+\sigma) w = u-u_h
\eeq
in $\Omega_{\de}$.
Then for all $w_h \in V_h$ we have
\bea \label{310}
\|u-u_h\|^2_{L^2(\Omega_{\de})} =&\left<(L^{*}+\sigma) w, u-u_h\right>_{H^{-1}(\Omega_{\de}), H^1(\Omega_{\de})} \\
=& \int_{\Omega_{\de}}\sigma w (u-u_h) + \left<w, L(u-u_h)\right>_{H^{1}(\Omega_{\de}), H^{-1}(\Omega_{\de})} \\
& -\left<w_h, L(u-u_h)\right>_{H^{1}(\Omega^h), H^{-1}(\Omega^h)} \\
\le& \sigma \|w\|_{L^2(\Omega_{\de})} \|u-u_h\|_{L^2(\Omega^h)}  \\
& +  (\Lambda+\nu\la)\|u-u_h\|_{H^{1,2}(\Omega^h)}
\|w-w_h\|_{H^{1,2}(\Omega^h)}.
\eea

(ii) Let $z \in L^2(\Omega_{\de})$, $\|z\|_{L^2(\Omega_{\de})}\le 1$ arbitrary. Then choose $\tilde z \in H^{1,2}_0(\Omega_{\de}) \cap H^{2,2}(\Omega_{\de})$ such that $L\tilde z = z$. From (\ref{230}) we deduce that
\beq
\|\tilde z\|_{H^{2,2}(\Omega_{\de})} \le c_3 \|z\|_{L^2(\Omega_{\de})}
\eeq
and get 
\bea 
\int_{\Omega_{\de}}wz =& \left<w, L \tilde z\right>_{H^1(\Omega_{\de}), H^{-1}(\Omega_{\de})} \\
=& \left< L^{*}w, \tilde z\right>_{H^{-1}(\Omega_{\de}), H^1(\Omega_{\de})} \\
=& \left<u-u_h-\sigma w, \tilde z\right>_{H^{-1}(\Omega_{\de}), H^1(\Omega_{\de})} \\
\le& \|u-u_h\|_{L^2(\Omega^h)}\|\tilde z\|_{L^2(\Omega_{\de})} + \sigma \|w\|_{L^2(\Omega_{\de})}\|\tilde z\| _{L^2(\Omega_{\de})} \\
\le& c_3 \|u-u_h\|_{L^2(\Omega^h)}\|z\|_{L^2(\Omega_{\de})} + c_3\sigma \|w\|_{L^2(\Omega_{\de})}\| z\| _{L^2(\Omega_{\de})}. 
\eea
Taking the supremum with respect to $z$ yields
\beq
\|w\|_{L^2(\Omega_{\de})} \le c_3\|u-u_h\|_{L^2(\Omega^h)}+ c_3\sigma \|w\|_{L^2(\Omega_{\de})}
\eeq
and therefore in view of (\ref{824})
\beq
\|w\|_{L^2(\Omega_{\de})} \le  2c_3 \|u-u_h\|_{L^2(\Omega^h)}.
\eeq
We use $w$ as a test function in (\ref{231}) and get
\beq
\|Dw\|_{L^2(\Omega_{\de})} \le cc_3\nu\|u-u_h\|_{L^2(\Omega_{\de})}
\eeq
hence $L^2$-estimates lead to 
\beq
\|w\|_{H^{2,2}(\Omega_{\de})} \le c_6\|u-u_h\|_{L^2(\Omega^h)} 
\eeq
with
\beq
c_6:=cc_3\Lambda\nu\frac{a_1}{\la^2}.
\eeq

(iii) We estimate 
\beq
\inf_{w_h \in V_h} \|w-w_h\|_{H^{1,2}(\Omega^h)}
\eeq
from above.
 
 Let $\tilde w$ be an extension of 
 \beq
 w_{ |\Omega_{\de}\backslash  \Omega_{-\de/2} }
 \eeq
 so that
 \beq
 \tilde w \in H^{2,2}_0(\Omega_{2\de}\backslash \Omega_{-4\de}), \quad \|\tilde w\|_{H^{2,2}(\Omega_{2\de}\backslash  \Omega_{-4\de })} \le
 c \|w\|_{H^{2,2}(\Omega_{\de}\backslash  \Omega_{-\de/2 })}. 
 \eeq
For $w_h \in V_h$ we have
 \bea
 \|w-w_h\|_{H^{1,2}(\Omega^h)} \le& \|w-\tilde w-w_h\|_{H^{1,2}(\Omega^h)} + \|\tilde w\|_{H^{1,2}(\Omega^h)}. 
 \eea
 Since $w-\tilde w \in H^{1,2}_0(\Omega^h)$ there holds
 \beq
 \inf_{w_h \in \Omega^h}\|w-\tilde w-w_h\|_{H^{1,2}(\Omega^h)} \le cc_6h \|u-u_h\|_{L^2(\Omega^h)}
 \eeq
 and, furthermore, chosing
 \bea \label{910}
 p  
 \begin{cases}
= (\frac{1}{2}-\frac{1}{n+1})^{-1}, &\quad \text{if } n>1, \\
 >2,  &\quad \text{if } n=1,
 \end{cases}
 \eea
 we have
 \bea
 \|\tilde w\|_{H^{1,2}(\Omega^h)}\le& \|\tilde w\|_{H^{1,2}(\Omega_{2\de}\backslash \Omega_{-4\de})} \\
 \le&c\de^{q} \|\tilde w\|_{H^{1,p}(\Omega_{2\de}\backslash \Omega_{-4\de})} \\
 \le& c\de^{q} \|\tilde w\|_{H^{2,2}(\Omega_{2\de}\backslash \Omega_{-4\de})} \\
 \le & cc_6\de^{q}\|u-u_h\|_{L^2(\Omega^h)},
 \eea
 where
 \beq \label{911}
 q=\frac{1}{2(p/2)^{*}},
 \eeq
 and hence
 \bea \label{351}
 \inf_{w_h\in V_h}\|w-w_h\|_{H^{1,2}(\Omega^h)} \le& cc_6(h+\de^{q})\|u-u_h\|_{L^2(\Omega^h)}.
 \eea
Combining (\ref{351}) and (\ref{310}) yields
\bea
\|u-u_h\|^2_{L^2(\Omega^h)}  & \le  
 c c_6(\Lambda+\nu\la)(h + \de^{q})\|u-u_h\|_{H^{1,2}(\Omega^h)}\|u-u_h\|_{L^2(\Omega^h)}
\eea
and therefore
\beq
\|u-u_h\|_{L^2(\Omega^h)}   \le  
c_7(h + \de^{q})\|u-u_h\|_{H^{1,2}(\Omega^h)}
\eeq
with
\beq \label{861}
c_7:= c c_6(\Lambda+\nu\la).
\eeq

(iv) We have for any $v_h \in V_h$
\bea
\frac{\la}{2}\|u-u_h\|^2_{H^{1,2}(\Omega^h)} \le& \left< L(u-u_h), u-u_h\right> + \nu^2\la\|u-u_h\|^2_{L^2(\Omega^h)} \\
=& \left< L(u-u_h), u-v_h\right> +\nu^2\la\|u-u_h\|^2_{L^2(\Omega^h)} \\
\le& 
(\Lambda+\nu\la)\|u-u_h\|_{H^{1,2}(\Omega^h)}\|u-v_h\|_{H^{1,2}(\Omega^h)} \\
&+ \nu^2\la c_7^2(h+\de^{q})^2
\|u-u_h\|^2_{H^{1,2}(\Omega^h)} 
\eea
and hence 
\beq
\|u-u_h\|_{H^{1,2}(\Omega^h)} \le c(\frac{\Lambda}{\la}+\nu)\|u-v_h\|_{H^{1,2}(\Omega^h)}.
\eeq

(v) Existence  of a FE solution  $u_h$ of (\ref{201}) follows in the usual way. Due to the quadratic structure of the corresponding system of linear equations, which determines $u_h$, we deduce existence from uniqueness, at which the latter is given in view of (\ref{241}) and (\ref{242}).
\end{proof}

\section{Proof of Theorem \ref{main_thm_endversion}} \label{Section_FE}
In this section we  will prove Theorem \ref{main_thm_endversion}, for it we obtain  the solution $u^{\epsilon}_h$ of (\ref{930}) as the unique fixed point of a map $T:V_h\rightarrow V_h$ in $\bar B^h_{\rho}$, cf. (\ref{912}), which will be defined in (\ref{204}). We show that in the situation of Theorem \ref{main_thm_endversion} we can choose $\be, \ga, \eta>0$ (and these values can be calculated explicitly) so that 
\beq \label{913}
\bar B^h_{\rho} \neq \emptyset,
\eeq
\beq \label{914}
\|Tw_h-Tv_h\|_{H^{1, \mu}(\bar \Omega^h)} \le c h^{\eta} \|w_h-v_h\|_{H^{1, \mu}(\bar \Omega^h)}
\quad \forall w_h, v_h \in \bar B^h_{\rho}
\eeq
and
\beq \label{915}
T({\bar B^{h}_{\rho}}) \subset {\bar B^{h}_{\rho}},
\eeq
i.e. Theorem \ref{main_thm_endversion} follows from Banach's fixed point theorem. 

(i) We define the map $T$.
 
 We define for $\epsilon>0$ and $z \in \mathbb{R}^n$
\beq
|z|_{\epsilon}:=f_{\epsilon}(z):=\sqrt{|z|^2+\epsilon^2}
\eeq
and denote derivatives of $f_{\epsilon}$ with respect to $z^i$ by $D_{z^i}f_{\epsilon}$.
There holds
\beq
D_{z^i}f_{\epsilon}(z)= \frac{z_i}{|z|_{\epsilon}}, \quad 
D_{z^i}D_{z^j}f_{\epsilon}(z)=\frac{\de_{ij}}{|z|_{\epsilon}} - \frac{z_iz_j}{|z|^3_{\epsilon}}.
\eeq
We define the operator $\Phi_{\epsilon}$ by
\beq
 \Phi_{\epsilon}: H^{1,2}_0(\Omega)\rightarrow H^{-1,2}_0(\Omega), \quad \Phi_{\epsilon}(v)=- D_i \left(\frac{D_i v}{|D v|_{\epsilon}}\right) - \frac{1}{|D v|_{\epsilon}^{\frac{1}{k}}}, 
\eeq
so that (\ref{regularized_levelset_pmcf}) can be written as
\beq
\Phi_{\epsilon}(u^{\epsilon})=0.
\eeq
We denote the derivative of $\Phi_{\epsilon}$ in $u^{\epsilon}$ by
\beq
L_{\epsilon}:=D\Phi_{\epsilon}(u^{\epsilon})
\eeq
and have for all $\varphi \in H^{1,2}_0(\Omega)$ that
\beq
L_{\epsilon}\varphi = -D_i\left(D_{z^i}D_{z^j}f_{\epsilon}(D u^{\epsilon})D_j \varphi\right)
+\frac{1}{k} f_{\epsilon}(D u^{\epsilon})^{-1-\frac{1}{k}}D_{z^j}f_{\epsilon}(D u^{\epsilon})D_j \varphi.
\eeq
We will apply the results of Section \ref{Section_linear_equations} to the linear differential operator $L=L_{\epsilon}$, where we consider---having (\ref{865}) in mind---$L_{\epsilon}$ (and also $M_{\epsilon})$ to be defined in $H^{1,2}_0(\Omega^h)$, $h>0$ small; one observes that $L_{\epsilon}$ has the structure (\ref{201}) and explicit values for the constants $\la, \Lambda, \nu, a_1, h_0$ in (\ref{864}), (\ref{869}) and (\ref{817}) can be obtained in terms of $\epsilon$ (this dependence is polynomial in $\epsilon$ and $\frac{1}{\epsilon}$)
using the results of Section \ref{higher_order_estimates} and (\ref{865}).
Let us denote the constant $c_8$ in Corollary \ref{823} adapted to the case of $L=L_{\epsilon}$ (as an operator defined in $H^{1,2}_0(\Omega^h)$) by $\tilde c_8$.
 
We define 
$T: V_h\rightarrow V_h$ by
\beq \label{204}
L_{\epsilon} (w_h-Tw_h) = \Phi_{\epsilon}(w_h), \quad w_h \in V_h.
\eeq

(ii) We check condition (\ref{913}).
 
 Let 
\beq
I_h:C^0(\bar \Omega^h)\rightarrow \tilde V_h,
\eeq
be the unique interpolation operator with
\beq
I_hu(p)=u(p)
\eeq
for all $u \in C^0(\bar \Omega^h)$ and $p\in N_h$, where
\beq
\tilde V_h:=\{w \in C^0(\bar \Omega^h): \forall_{T\in \mathbb{T}_h}w|_T \text{ polynom of degree} \le 2\}
\eeq
and
\bea
N_h:=\left\{ p \in \bar \Omega^h: p \text{ vertex or midpoint of an edge  of a tetraeder } T\in \mathbb{T}_h\right\}.
\eea
We have
\beq \label{866}
\|u-I_hu\|_{H^{1,\infty}(\Omega^h)} \le ch^{2}\|u\|_{C^{3}(\bar \Omega^h)} \quad \forall \; u \in C^{3}(\bar \Omega ^h), 
\eeq
define $z_h \in \tilde V_h$  by
\beq
z_h (p)= \begin{cases}I_hu^{\epsilon}(p), \quad &\text{if }p \in N_h \cap \partial \Omega^h, \\
0, \quad &\text{if }p \in N_h \backslash \partial \Omega^h \end{cases}
\eeq
and set
\beq \label{883}
\tilde u^{\epsilon}:= I_hu^{\epsilon}-z_h.
\eeq
Then $\tilde u^{\epsilon}\in V_h$ and for all $1\le q \le \infty$ 
\beq \label{901}
\|\tilde u^{\epsilon}-u^{\epsilon}\|_{H^{1,q}(\Omega^h)} \le c h^{1+\frac{1}{q}}\|u^{\epsilon}\|_{C^{3}(\bar \Omega^h)}, 
\eeq
which follows from 
\bea
\|z_h\|_{C^0(\bar \Omega^h)} \le ch^2, \quad \|Dz_h\|_{L^{\infty}(\Omega^h)} \le ch
\eea
and that the support of $ z_h$ lies in a boundary strip of measure $\le ch$. 

We conclude $\tilde u^{\epsilon} \in \bar B^h_{\rho}$ provided $\be, \ga>0$ are sufficiently large.

(iii) We check condition (\ref{914}).

Let $q>n+1$ and $v_h, w_h \in \bar B^{h}_{\rho}$, $\xi_h = v_h-w_h$, $\al(t)=w_h + t\xi_h$, $0\le t \le 1$, then using (\ref{204}) we conclude

\beq \label{501}
L_{\epsilon} (Tv_h-Tw_h)  = L_{\epsilon}\xi_h + \Phi_{\epsilon}(w_h)-\Phi_{\epsilon}(v_h).
\eeq
The right-hand side of (\ref{501}) is of the form $D_if^{i}+g$ with
\bea \label{868}
f^i =& D_{z^i}f_{\epsilon}(D v_h)-D_{z^i}f_{\epsilon}(D w_h)-D_{z^i}D_{z^j}f_{\epsilon}D_j \xi_h \\
=& \int_0^1\left(D_{z^j}D_{z^i}f_{\epsilon}(D \al(t))- D_{z^i}D_{z^j}f_{\epsilon}\right) D_j \xi_h
\eea
and
\bea \label{867}
g=& \frac{1}{k}f_{\epsilon}^{-1-\frac{1}{k}}D_{z^j}f_{\epsilon}D_j \xi_h + f_{\epsilon}(D v_h)^{-\frac{1}{k}}- f_{\epsilon}(D w_h)^{-\frac{1}{k}} \\
=& \frac{1}{k}\int_0^1 \left( f_{\epsilon}^{-1-\frac{1}{k}}D_{z^j}f_{\epsilon}-
f_{\epsilon}^{-1-\frac{1}{k}}(D \al(t))D_{z^j}f_{\epsilon}(D \al(t))\right)D_j \xi_h.
\eea
We have
\bea
\|Dw_h-Du^{\epsilon}\|_{L^{\infty}(\Omega^h)} \le&
\|Dw_h-DI_h u^{\epsilon}\|_{L^{\infty}(\Omega^h)}+\|DI_hu^{\epsilon}-Du^{\epsilon}\|_{L^{\infty}(\Omega^h)}  \\
\le& ch^{-\frac{n+1}{\mu}}(\|Dw_h-Du^{\epsilon}\|_{L^{\mu}(\Omega^h)}\\
&+\|Du^{\epsilon}-DI_h u^{\epsilon}\|_{L^{\mu}(\Omega^h)}) +ch^2\|u^{\epsilon}\|_{C^3(\bar \Omega^h)}\\
\le& ch^{-\frac{n+1}{\mu}} (\rho + h^{2}\|u^{\epsilon}\|_{C^3(\bar \Omega^h)}),
\eea
where we used an inverse estimate and (\ref{866}). We estimate the integrals in (\ref{867}) and (\ref{868}) by mean value theorem and get with a constant $c_9:=c_9(\epsilon, k)$ 
\bea
\|f^i\|_{L^{q}(\Omega^h)} + \|g\|_{L^2(\Omega^h)} \le& c c_9 h^{\frac{n+1}{q}-2\frac{n+1}{\mu}} (\rho + h^{2})\rho,
\eea
Therefore we get
\bea \label{870}
\|T{v_h}-Tw_h\|_{H^{1,\mu}(\Omega^h)} \le& h^{\frac{n+1}{\mu}-\frac{n+1}{2}} \|T{v_h}-Tw_h\|_{H^{1,2}(\Omega^h)} \\
\le& c \tilde c_8 c_9 h^{\frac{n+1}{q}-\frac{n+1}{\mu}-\frac{n+1}{2}} (\rho + h^{2}) \rho,
\eea
in view of Corollary \ref{823}.

Assuming $n=1$ and $2<q<\mu$ we have
\beq
1+\frac{2}{\mu}-\frac{2}{q} <1< \de
\eeq
and hence (\ref{914}) holds provided $\be>0$ is sufficiently large.

(iv) We check condition (\ref{915}).

Let $q>n+1$ and $w_h \in V_h$. We have
\bea
\|Tw_h-u^{\epsilon}\|_{H^{1, \mu}(\Omega^h)} \le& 
\|Tw_h-T\tilde u^{\epsilon}\|_{H^{1, \mu}(\Omega^h)} +\|T\tilde u^{\epsilon}-\tilde u^{\epsilon}\|_{H^{1, \mu}(\Omega^h)} \\
&+\|\tilde u^{\epsilon}-u^{\epsilon}\|_{H^{1, \mu}(\Omega^h)}
\eea 
We estimate the three terms on the right-hand side of this inequality separately and get
\beq
\|\tilde u^{\epsilon}-u^{\epsilon}\|_{H^{1, \mu}(\Omega^h)} \le c h^{1+\frac{1}{\mu}}\|u^{\epsilon}\|_{C^3(\bar \Omega^h)}
\eeq
and
\bea \label{916}
\|Tw_h-T\tilde u^{\epsilon}\|_{H^{1, \mu}(\Omega^h)} \le&
c h^{\eta} \|w_h-\tilde u^{\epsilon}\|_{H^{1, \mu}(\Omega^h)} \\
\le& c h^{\eta} \|w_h- u^{\epsilon}\|_{H^{1, \mu}(\Omega^h)} 
+c h^{\eta} \|u^{\epsilon}-\tilde u^{\epsilon}\|_{H^{1, \mu}(\Omega^h)} \\
\le& ch^{\eta}\rho + c h^{\eta+1+\frac{1}{\mu}}\|u^{\epsilon}\|_{C^3(\bar \Omega^h)}.
\eea
Let $\xi = u^{\epsilon}-\tilde u^{\epsilon}$, $\al (t)= \tilde u^{\epsilon} + t \xi$, $0\le t\le1$.
We have in $\Omega^h$
\bea
L_{\epsilon}\left(\tilde u^{\epsilon}-T(\tilde u^{\epsilon})\right) =& \Phi_{\epsilon}(\tilde u^{\epsilon}) \\
=& \Phi_{\epsilon}(\tilde u^{\epsilon}) -\Phi_{\epsilon}(u^{\epsilon}) + \Phi_{\epsilon}(u^{\epsilon}).
\eea
and the right-hand side of this equation is of the form $D_if^i+g$ with
\bea
f^i =& -D_{z^i}f_{\epsilon}(D \tilde u^{\epsilon}) + D_{z^i}f_{\epsilon}(D u^{\epsilon}) \\
=& \int_0^1 D_{z^j}D_{z^i}f_{\epsilon}(D \al(t))D_j\xi
\eea
and
\bea
g =& -f_{\epsilon}(D \tilde u^{\epsilon})^{-\frac{1}{k}}+
f_{\epsilon}(D u^{\epsilon})^{-\frac{1}{k}} + \Phi_{\epsilon}(u^{\epsilon})  \\
=& \int_0^1 D_{z^i}f_{\epsilon}(D \al(t))D_i\xi + \Phi_{\epsilon}(u^{\epsilon}) 
\eea
We have
\bea
\|f^i\|_{L^{q}(\Omega^h)} + \|g\|_{L^2(\Omega^h)} \le& c_{10} \|D \xi\|_{L^{q}(\Omega^h)} +c_{11}h^3
\eea
with $c_{10}:=c_{10}(\epsilon)$ suitable and $c_{11}:= \sup_{\Omega^h} |D(\Phi_{\epsilon}(u^{\epsilon}))|$.
Finally, we get
\bea
\|\tilde u^{\epsilon}-T(\tilde u^{\epsilon})\|_{H^{1,\mu}(\Omega^h)} \le& c\tilde c_8 h^{\frac{n+1}{\mu}-\frac{n+1}{2}}(c_{10}h^{1+\frac{1}{q}}\|u^{\epsilon}\|_{C^3(\bar \Omega^h)}+c_{11}h^3) \\
\le& c_{12}h^{\frac{n+1}{\mu}-\frac{n+1}{2}+1+\frac{1}{q}}
\eea
with $c_{12}=c_{12}(\epsilon)$.

To allow for (\ref{915}) in case $n=1$ it is sufficient to have
\beq
\de < \frac{2}{\mu}+\frac{1}{q},
\eeq
which holds for $q>2$ close to 2, and $\be>0$ sufficiently large.

\vspace{0.5cm}

{\bf Acknowledgment:} The author would like to thank Guy Barles and Andreas Prohl for discussion and the SFB-Transregio 71 of the German Science Foundation (DFG) for funding.

\end{document}